\documentclass[10pt]{article}
\font\smallit=cmti10

\usepackage{amssymb,latexsym,amsmath,epsfig,amsthm}

\usepackage{tikz}

\usepackage{color}

\makeatletter

\renewcommand\section{\@startsection {section}{1}{\z@}
{-30pt \@plus -1ex \@minus -.2ex}
{2.3ex \@plus.2ex}
{\normalfont\normalsize\bfseries\boldmath}}

\renewcommand\subsection{\@startsection{subsection}{2}{\z@}
{-3.25ex\@plus -1ex \@minus -.2ex}
{1.5ex \@plus .2ex}
{\normalfont\normalsize\bfseries\boldmath}}

\renewcommand{\@seccntformat}[1]{\csname the#1\endcsname. }

\makeatother

\newtheorem{theorem}{Theorem}

\newtheorem{corollary}{Corollary}

\theoremstyle{definition}

\usepackage{graphicx}
\usepackage{amssymb}
\usepackage{float}

\theoremstyle{remark}
\numberwithin{equation}{section}

\DeclareMathOperator{\li}{li}

\begin{document}

\begin{center}
\uppercase{\bf \boldmath Lines in the prime number graph}
\vskip 20pt

{\bf Scott~Duke~Kominers}\\
{\smallit Harvard University, Cambridge, Massachusetts, USA}\\
{\tt kominers@fas.harvard.edu}
\vskip10pt
{\bf Rudi~Mrazovi\'c}\\
{\smallit University of Zagreb, Zagreb, Croatia}\\
{\tt Rudi.Mrazovic@math.hr}
\vskip10pt
{\bf Carl~Pomerance}\\
{\smallit Dartmouth College, Hanover, New Hampshire, USA}\\
{\tt carlp@math.dartmouth.edu}
\vskip 10pt
{\bf Patrick~Sol\'e}\\
{\smallit I2M, (CNRS, Aix-Marseille University), Marseille, France}\\
{\tt patrick.sole@telecom-paris.fr}
\end{center}
\vskip 20pt

\centerline{\bf Abstract}
The prime number graph is the set of points $(n,p_n)$ where $p_n$ denotes the $n^{\text{th}}$ prime. Let $L(n)$ be the minimum number of straight lines needed to cover the first $n$ points in this set.
Let $B(n)$ be the largest number of points $(k,p_k)$ with $k\le n$ covered by a single line. 
Recently Sloane conjectured that $L(n) = O(n/\log n)$.  We prove a much stronger bound,
as well as upper and lower estimates for $B(n)$.  Our proofs use the Prime Number Theorem with remainder
and are considerably improved with the assumption of the Riemann Hypothesis.
\noindent
\vskip 20pt
{\small {\bf AMS Subject Classification:} 52C10, 11A41, 11N05.}

{\small {\bf Keywords:} Prime Number Theorem, prime number graph, awkward prime.}

\pagestyle{myheadings}
\thispagestyle{empty}
\baselineskip=12.875pt
\vskip 30pt
\section{Introduction}
Let $p_1,p_2,\dots$ denote the sequence of primes.
A {\bf prime point} is a point of the plane of the form $(k,p_k)$ for some $k.$ This  graphical representation of the primes was considered
in \cite{P}.  It is interesting to look at sets of these prime points that are collinear, such as
\[
(6, 13),\, (7, 17), \,(10, 29),\, (12, 37),\, (13, 41), \,(16, 53), \,(18, 61), \,(21,73)
\]
which are all on the line $y=4x-11$.
Let $L(n)$ be the minimum number of lines needed to cover the first $n$ prime points. 
For example, $L(2)=1$ and $L(3)=2.$

\begin{figure}[ht]
    \centering
    \begin{tikzpicture}
        [xscale=0.5, yscale=0.08, font=\small]
    
        % axes
        \draw[->] (0,0)--(24,0) node[below] {$n$};
        \draw[->] (0,0)--(0,89) node[left] {$p_n$};
    
        % the 5 covering lines
   %     \draw[thin]
    %        (1,{1+1})--(24,{24+1});     % p_n = n+1
            
        \draw[thin]
           (1,2)--(24,89);   %p_n = (87/23)n-41/23
               
        \draw[thin]
            (1,{2*1-1})--(24,{2*24-1}); % p_n = 2n-1
    
        \draw[thin]
            (2.5,1)--(24,{4*24-9}); % p_n = 4n-9
    
        \draw[thin]
            (3,1)--(24,{4*24-11}); % p_n = 4n-11
    
        \draw[thin]
            (3.5,1)--(24,{4*24-13}); % p_n = 4n-13
    
        % points (n, p_n), n = 1,...,24
        \foreach \n/\p in {
            1/2,
            2/3,
            3/5,
            4/7,
            5/11,
            6/13,
            7/17,
            8/19,
            9/23,
            10/29,
            11/31,
            12/37,
            13/41,
            14/43,
            15/47,
            16/53,
            17/59,
            18/61,
            19/67,
            20/71,
            21/73,
            22/79,
            23/83,
            24/89
        } {
            \draw[fill=black]
                ({\n},{\p}) circle[shift only, radius=2.0pt];
        }
    \end{tikzpicture}
    \caption{First $24$ points of the prime number graph covered by $L(24) = 5$ lines.}
\end{figure}
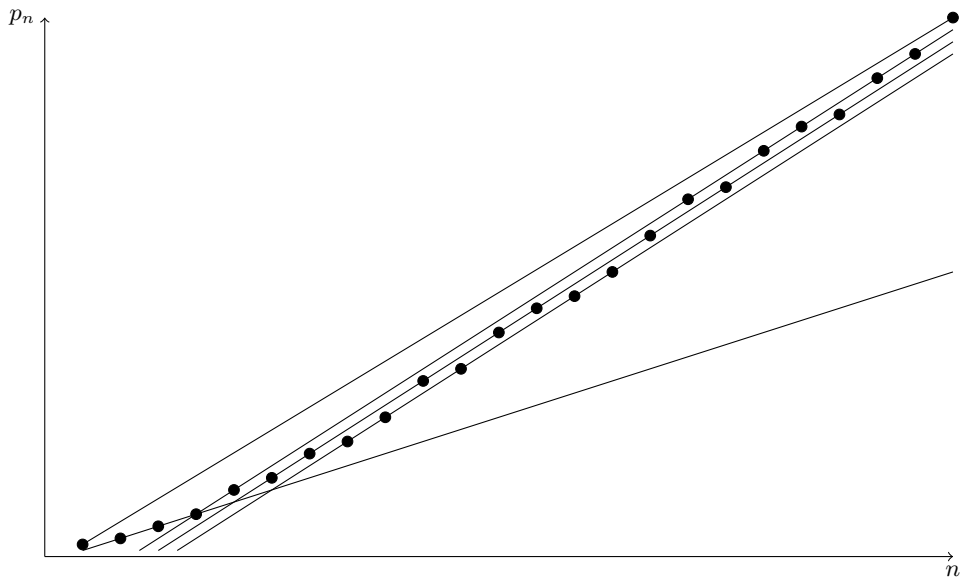

A prime $p_n$ is {\bf awkward} if $L(n)>L(n-1).$ These concepts were introduced in recent Numberphile videos \cite{S1}, \cite{S2}.
See \cite{O1} for numerics of $L(n)$ (called there $a(n)$)  for small $n$, and see  \cite{O2} for the list of the first awkward primes.

In this note we study how the function $L(n)$ behaves for large $n$.
 Since the primes have asymptotic
density 0, it is clear that no line can contain infinitely many prime points $(k,p_k)$.
We introduce the function $B(n)$ which is the largest number of prime points among the first $n$ of them   covered by a single line, and derive  upper and lower bounds for it.  Our arguments for $L(n)$ and
$B(n)$ are based on the Prime Number Theorem with remainder; we therefore also obtain a direct strengthening if we assume the Riemann Hypothesis.

The material is arranged as follows. The next section recalls some known results on the topic of the prime number graph and the Prime Number Theorem with remainder. Sections~\ref{sec-L(n)} and~\ref{sec-B(n)} study the functions $L(n)$ and $B(n)$, respectively.  Section~\ref{sec-conclusion} presents concluding remarks and underlines some challenging open problems.

\section{Background results} \label{sec-background}

We begin by recording an observation that is originally due to Erd\H{o}s and quoted without proof in~\cite{P}.

\begin{theorem}
\label{th-oldtheorem}
 For any positive integer $k$, almost all prime points $(n,p_n)$ lie on a line with $k$ other prime points.
 That is, the set of primes for which this is so has relative density $1$ in the set of primes.
\end{theorem}
 From there, we deduce an asymptotic upper bound on $L(n).$

\begin{corollary}
\label{cor-oldtheorem} 
 For $n \to \infty$ we have $L(n)=o(n).$ 
\end{corollary}

In the sequel, we present a detailed proof of a more quantitative version of Theorem~\ref{th-oldtheorem} and Corollary~\ref{cor-oldtheorem}. We also derive a lower bound for $B(n)$, using a proof strategy strongly based on the proof of
\cite[Theorem 4.1]{P}.

We note that based on the experimental data in \cite{O1}, Sloane conjectured in \cite{S1} that $L(n)=O(\frac{n}{\log n}).$ This in particular motivates Theorem~\ref{th-L(n)} of the next section, where we in fact obtain a bound stronger than Sloane's conjecture.

As mentioned in the Introduction, our proofs strongly use the Prime Number Theorem with
remainder.  In particular, let
\[
\li(x)=\int_0^x\frac{dt}{\log t}
\]
denote the logarithmic integral function (where the principal value is taken for the singularity
at $t=1$).  Then $\li(x)\sim x/\log x\sim\pi(x)$ as $x\to\infty$, but the $\li(x)$ approximation
to $\pi(x)$ is much more accurate.  In particular, we have
\begin{equation}
\label{eq-lipnt}
|\pi(x)-\li(x)|\le x/\exp(c(\log x)^{3/5}(\log\log x)^{-1/5})
\end{equation}
for a positive constant $c$ and all large $x$.  If the Riemann Hypothesis is assumed, then
\eqref{eq-lipnt} improves to
\[
|\pi(x)-\li(x)|\le x^{1/2}\log x
\]
for all $x\ge2$. (See \cite{K} and \cite{MV}; and see also \cite{PT} for
an asymptotically weaker, but numerically explicit version of \eqref{eq-lipnt}.)
In general, let $R(x)$ be any smooth function with $R'>0$ and $R''<0$ and
\begin{equation}
\label{eq-pnt}
|\pi(x)-\li(x)|\le R(x)
\end{equation}
for all sufficiently large values of $x$.  Our subsequent results are all stated in
terms of $R(x)$.
In fact, our results hold for any increasing integer sequence
whose counting function satisfies \eqref{eq-pnt}.

\section{The quantitative Erd\H{o}s observation and awkward primes}\label{sec-L(n)}

In this section we give a quantitative proof of Theorem \ref{th-oldtheorem}.  In particular, we will prove the following.
\begin{theorem}
\label{th-newversion}
Let $N(x)$ denote the number of $n\le x$ such that
 $(n,p_n)$ lies on a line with at least
$(n/R(n))^{1/4}$  prime points in total. Then $N(x)\sim x$ as $x\to\infty$.
\end{theorem}
We prove Theorem \ref{th-newversion} as a corollary of the following result.
\begin{theorem}
\label{th-L(n)}
We have $L(n)=O(n^{3/4}R(n)^{1/4}/(\log n)^{1/2})$.
\end{theorem}

\begin{proof}
Let $Q$ be a large integer and
consider the Farey sequence of
level $Q$.  Let $a/b, a'/b'$ be two consecutive terms.  Then
\[
a'/b'-a/b = 1/bb',~~b,b'\le Q,~~b+b'>Q,~~\gcd(b,b')=1.
\]
Let $k$ be a large integer, and let
 $u=e^{k+a/b},\,u'=e^{k+a'/b'}$, with $I$ the interval $(u,u']$.  We consider \textbf{inverse prime points} $(p_n,n)$
with $p_n\in I$.   The length $|I|$ of $I$ has
\[
|I|=u'-u=e^{k+a'/b'}-e^{k+a/b}=e^{k+a/b}(e^{1/bb'}-1)=(1+o(1))u/bb', \quad Q\to\infty.
\]
Let $P$ denote the parallelogram bounded by the vertical lines $x=u,\, x=u'$
and the lines with slope $1/\log u=1/(k+a/b)$ through $(u, \li(u)-w)$ and $(u,\li(u)+w)$,
where
\[
w=\frac{|I|^2}{u\log^2u}=(1+o(1))\frac u{(bb')^2\log^2u}.
\]
These lines gain $|I|/\log u$ on the interval $I = (u,u']$.  This is about the same gain as
$\li(x)$ on the interval.  Note that by Taylor's theorem,
\[
\li(u')-\li(u)=\frac{|I|}{\log u}-\left(\frac12+o(1)\right)\frac{|I|^2}{u\log^2 u}=\frac{|I|}{\log u}-\left(\frac12+o(1)\right)w,
\]
and hence the region
\[
    |y-\li(x)|\le w/4,~~x\in I
\]
lies wholly in $P$.
For a given large number $k$ we want to choose $Q$ as large as possible so that the graph of $y=\pi(x)$ for $x\in I$ also lies in $P$.
Since $bb'<Q^2$, we have $w/4 \geq e^k / (8Q^4 k^2)$ for $k$ large, and so by \eqref{eq-pnt} this will be accomplished if we take
\[
Q=\left\lfloor \left(\frac{e^k}{8R(e^{k+1})k^2}\right)^{1/4}\right\rfloor.
\]

We can also do a similar construction by using $u'$ instead of $u$ to determine the slope.
Namely, let $P'$ be the parallelogram bounded by
$x=u,\,x=u'$ and the lines with slope $1/\log u'$ through $(u',\li(u')- w)$ and $(u',\li(u')+ w)$.
With the same choice of $Q$, the prime count $y=\pi(x)$ for $x\in I$ also lies in $P'$ for $k$ large.

Hence all of the inverse prime points $(p_n,n)$ with $p_n\in I$ lie in both $P$ and $P'$.  

If $b<b'$, we
consider those lines with slope $1/\log u=1/(k+a/b)=b/(bk+a)$ that pass through a lattice point in $P$.
Equations of these lines are $(bk+a)y - bx = C$ for different integers $C$,
and so there are at most
\[
(2+o(1))w \cdot (bk+a)\sim 2wb\log u\sim\frac{2u}{bb'^2\log u}\ll\frac{e^k}{bb'^2k}
\]
of them.  If $b>b'$, we take those lines with slope $1/\log u'=b'/(b'k+a')$ which pass through a
lattice point in $P'$; there are $\ll e^k/(b^2b'k)$ of them.
So if we consider all of the lines appearing in this argument for primes in $(e^k,e^{k+1}]$,
the number of them is bounded by a constant times
\begin{equation}
\label{eq-linecount}
\frac{e^k}k\sum\frac{\min\{b,b'\}}{(bb')^2},
\end{equation}
where the sum is over the full Farey dissection of level $Q$.  If $b,b'>Q/2$, then the summand
in \eqref{eq-linecount} is of magnitude $1/Q^3$, and there are fewer than $Q^2$ such pairs,
so the contribution is bounded by $1/Q$.  So, assume $\min\{b,b'\}\le Q/2$.
The contribution to the sum in \eqref{eq-linecount} is at most
\[
2\sum_{b\le Q/2}\frac1b\sum_{Q-b<b'\le Q}\frac1{b'^2}
<2\sum_{b\le Q/2}\frac1b\left(\frac1{Q-b}-\frac1Q\right)
=2\sum_{b\le Q/2}\frac1{(Q-b)Q}\ll\frac1Q.
\]
Thus, the sum in \eqref{eq-linecount} is $O(1/Q)$ and so all of the inverse prime points $(p_n,n)$ with $p_n\in(e^k,e^{k+1}]$ are covered by 
$O(e^k/kQ)$ lines.  With our choice for $Q$
and noting that $R(e^k)\sim R(e^{k+1})$, we have these inverse prime points covered
by $O(e^{3k/4}R(e^k)^{1/4}/k^{1/2})$ lines.  Summing this for $k\le K-1$, we have that the total number of lines that contain
some $(p_n,n)$ for $n\le e^K$ is $O(e^{3K/4}R(e^K)^{1/4}/K^{1/2})$.  Thus, if $n\in(e^{K-1},e^K]$, then
$L(n)=O(n^{3/4}R(n)^{1/4}/(\log n)^{1/2})$.  This completes the proof.
\end{proof}

As a corollary we obtain Theorem \ref{th-newversion}.

\begin{proof}[Proof of Theorem \ref{th-newversion}]
Consider the $L(n)$ lattice lines that cover all of the
points $(p_j,j)$ for $j\le n$.  Those lines that cover fewer than $(n/R(n))^{1/4}$
inverse prime points together cover $O(n/(\log n)^{1/2})$ points.  This leaves
still asymptotically all $n$ inverse prime points, where each such point
is contained in a line with at least  $(n/R(n))^{1/4}$ other inverse prime points.
 \end{proof}

 \begin{theorem}
 \label{th-awkward}
 The number of awkward primes among the first $n$ primes is $L(n)$.
 Thus, the reciprocal sum of the awkward primes is finite.
 \end{theorem}
 \begin{proof}
 Let $L(0)=0$.
 For each positive integer $j$, we have $L(j)-L(j-1)$ equal to 0 or 1, where the value 1 occurs
 if and only if $p_j$ is awkward
 (because adding the~$j^{\text{th}}$ point can always be handled by adding one line). So, we have
 \[
 L(n)=\sum_{j\le n}(L(j)-L(j-1)).
 \] 
 Thus, it is clear then that $L(n)$ is the number of awkward primes $p_j$ with $j\le n$.
 Theorem \ref{th-newversion}, together with a partial summation argument, then
shows that their reciprocal sum is finite.
 \end{proof}

\section{The function $B(n)$}\label{sec-B(n)}

We know by \cite[Theorem~4.1]{P} and Theorem \ref{th-newversion}  that $B(n)$ is not bounded above. 
In fact, we have the following estimate.
\begin{theorem}
 \label{th-B(n)}
 There is a positive constant $c_1$ such that for all large $n$ we have
 $B(n)\ge c_1\sqrt{n/R(n)}/\log n$.
 \end{theorem}
\begin{proof}
First, we note that there is a simple combinatorial relation connecting the functions $L$ and $B$, namely
\begin{equation}
\label{eq-combo}
L(n)B(n)\ge n.
\end{equation}
This is immediate by considering a covering of the first $n$ prime points by $L(n)$ line segments. Each line contains at most $B(n)$ prime points, which then gives \eqref{eq-combo}.  Thus, from Theorem \ref{th-L(n)} we have
$B(n)=\Omega((n/R(n))^{1/4}(\log n)^{1/2})$.  But if we use the proof of Theorem \ref{th-L(n)}, then we can do
better.  In that proof, we used the Farey dissection of level $Q$ to obtain a
dissection of the interval $(e^k,e^{k+1}]$.   Now we use only the first (and longest) piece of the
dissection---namely, $(e^k,e^{k+1/Q}]$.  The length of this interval is $\sim e^k/Q$ and $w\sim e^k/(Qk)^2$.
Now the constraint on $Q$ being large is somewhat relaxed and we can take $Q$ as an integer
near $\sqrt{e^k/R(e^k)}/k$.  The number of primes in the interval is $\sim e^k/kQ$ and the number
of lines that cover them is $O(wk)=O(e^k/Q^2k)$.  Thus, the average number of inverse prime
points per line is $\Omega( Q)$.  So there is a positive constant $c_1$ such that $B(n)\ge c_1\sqrt{n/R(n)}/\log n$.
\end{proof}

We remark that using the interval $(e^k,e^{k+1/k}]$ to show that some lines have many
prime points was used in the proof of \cite[Theorem 4.1]{P}.

A trivial upper bound for $B(n)$ is $n$, of course.  We can do considerably better.
\begin{theorem}
\label{th-Bbound}
For $n$ sufficiently large, we have
\[
B(n)=O(\sqrt{nR(n)}).
\]
\end{theorem}
\begin{proof}
 From the discussion above, we may assume that $R(x)=o(\li(x))$ as $x\to\infty$ and that
 both functions $y=\li(x)+R(x)$ and $y=\li(x)-R(x)$ are smooth, strictly increasing, and strictly 
 concave down.
 Thus, a line may intersect these two curves in at most two points each.  In fact, a line can
 intersect the region $|y-\li(x)|\le R(x)$ at most twice, i.e., either for one bounded interval $I$
 on the positive $x$-axis or for two disjoint bounded intervals.  A calculation shows that the
 length of such an  interval is $O((xR(x))^{1/2}\log x)$.   Since the number of primes in such
 an interval is $O((xR(x))^{1/2})$, the theorem follows.
 \end{proof}

Assuming the Riemann Hypothesis (RH) we can give a more explicit upper bound on $B(n)$, and provide a similar refinement for $L(n)$.

\begin{corollary} 
 \label{cor-RH}
Under RH we have for $n$  large that
\[
n^{1/4}/(\log n )^{3/2} \ll B(n)\ll n^{3/4}(\log n )^{1/2},
\]
and 
\[
n^{1/4}/(\log n)^{1/2}\ll L(n)\ll n^{7/8}/(\log n)^{1/4}.
\]
\end{corollary}
\begin{proof}
 As we noted in Section \ref{sec-background}, we can take 
 $R(x)=\sqrt{x} \log x$ under RH. The results then follow from Theorems \ref{th-newversion},
\ref{th-B(n)}, and \ref{th-Bbound},  and \eqref{eq-combo}.
\end{proof}

\section{Conclusion and open problems}\label{sec-conclusion}

In this note we have studied the covering properties of line segments in the prime number graph. We have derived an asymptotic upper bound for the minimum size of a cover, as well as estimates for the largest number of prime points on a single segment.  Our estimates seem far from optimal, as is also suggested
from the numerical work in \cite{O1}, \cite{O2}. 
In particular, regarding the functions $L(n)$ and $B(n)$ it would be nice to reduce the huge gaps between
our upper and lower bounds.

Numerical experiments indicate that $L(n)$ is achieved by lines many of which are parellel to each other.
Our proof of Theorem \ref{th-L(n)} also utilizes such sets of lines.
This motivates considering the quantity $L_\mathrm{np}(n)$---the minimal number of lines that cover the first $n$ points of the prime number graph and have pairwise different slopes.
It seems that even proving that $L_\mathrm{np}(n) = o(n)$ is nontrivial.

It seems straightforward to generalize our results to primes in a fixed residue class, where the
Extended Riemann Hypothesis (namely, the RH for Dirichlet L-functions) plays a role.  Likewise
one can also look at primes of a particular splitting type in an algebraic number field, using
the Chebotarev density theorem.
More interestingly,
one can ask about general increasing integer sequences.
For example, it
follows from~\cite{P2} that if $a_1<a_2<\dots$ is a sequence of positive integers
with $\lim\inf a_n/n<\infty$, then for every $k$ there are $k$ collinear points $(n,a_n)$.
Is this true under the weaker hypothesis that $\sum 1/a_n=\infty$?   Given $x$, what is the largest
number $n$ for which there is an integer sequence $0<a_1<a_2<\dots<a_n\le x$ such that no
three points $(j,a_j)$ are collinear?  This holds for the $\lfloor\sqrt{x}\rfloor$ squares in $[1,x]$,
can one do better?
For similar problems more general than the graph of an
integer sequence, see \cite{B}.

\subsection*{Acknowledgments}
S.D.K.\ is a Research Partner at a16z crypto. This work was conducted while he was visiting the
Technological Innovation, Entrepreneurship, and Strategic Management (TIES)
Group at the MIT Sloan School of Management; he greatly appreciates their hospitality.

R.M.\ was supported by the Croatian Science Foundation under the project no.\ HRZZ-IP-2022-10-5116 (FANAP) and by the European Union -- NextGenerationEU through the National Recovery and Resilience Plan 2021-2026 Institutional grant of University of Zagreb Faculty of Science (IK IA 1.1.3.\ Impact4Math).

 \end{document}